\documentclass[12pt]{amsart}
\font\bbbld=msbm10 scaled\magstephalf

\newcommand{\bi}{\bar{i}}
\newcommand{\bj}{\bar{j}}
\newcommand{\bk}{\bar{k}}
\newcommand{\bl}{\bar{l}}
\newcommand{\bm}{\bar{m}}
\newcommand{\bn}{\bar{n}}
\newcommand{\bp}{\bar{p}}
\newcommand{\bq}{\bar{q}}

\newcommand{\bz}{\bar{z}}

\newcommand{\bM}{\bar{M}}

\newcommand{\bpartial}{\bar{\partial}}

\newcommand{\fa}{\mathfrak{a}}
\newcommand{\fg}{\mathfrak{g}}

\newcommand{\fRe}{\mathfrak{Re}}

\newcommand{\bfC}{\hbox{\bbbld C}}

\newcommand{\bfR}{\hbox{\bbbld R}}
\newcommand{\bfS}{\hbox{\bbbld S}}

\newcommand{\cT}{\mathcal{T}}

\newcommand{\tr}{\mbox{tr}}

\newcommand{\ol}{\overline}
\newcommand{\ul}{\underline}

\newtheorem{theorem}{Theorem}[section]
\newtheorem{lemma}[theorem]{Lemma}
\newtheorem{proposition}[theorem]{Proposition}

 \theoremstyle{definition}

\theoremstyle{remark}
\newtheorem{remark}[theorem]{Remark}

\numberwithin{equation}{section}

\begin{document}
\setlength{\baselineskip}{1.2\baselineskip}

\title[Complex Monge-Amp\`ere Type Equation]
{The Dirichlet Problem for a Complex Monge-Amp\`ere Type Equation \\
    on Hermitian Manifolds}
\author{Bo Guan}
\address{Department of Mathematics, Ohio State University,
         Columbus, OH 43210}
\email{guan@math.osu.edu}
\author{Qun Li}
\address{Department of Mathematics and Statistics, Wright State University,
         Dayton, OH 45435}
\email{qun.li@wright.edu}
\thanks{Research of both authors were supported in part by
NSF grants.}

\begin{abstract}
We are concerned with fully nonlinear elliptic
equations on Hermitian manifolds and search for technical tools to overcome
difficulties in deriving {\em a priori} estimates which arise due to
the nontrivial torsion and general (non-pseudoconvex) boundary data.
We present our methods, which work for more general equations, by considering
a specific equation which resembles the complex Monge-Amp\`ere equation in
many ways but with crucial differences. Our work is motivated by recent
increasing interests in fully nonlinear equations on complex manifolds
from geometric problems.



\end{abstract}

\maketitle

\bigskip

\section{Introduction}
\label{gblq-I}
\setcounter{equation}{0}
\medskip

In this paper we continue our study in \cite{GL10, GL12} on
fully nonlinear elliptic equations of Monge-Amp\`ere type
on complex manifolds.
Let $(M^n, \omega)$ be a compact Hermitian manifold of
dimension $n \geq 2$ with smooth boundary $\partial M$
and $\bM = M \cup \partial M$.
Let $\chi$ be a smooth real $(1,1)$ form on $\bM$.
Define for a function $u \in C^2 (M)$
\[ \chi_u = \chi + \frac{\sqrt{-1}}{2} \partial \bpartial u. \]
We consider the Dirichlet problem
 \begin{equation}
\label{CH-I10}
\left\{ \begin{aligned}
 & \chi_u^n =  \psi \chi_u \wedge \omega^{n-1}
\;\; \mbox{in $\bM$}, \\
 &        u = \varphi \;\; \mbox{on $\partial M$}.
\end{aligned} \right.
\end{equation}
This is a fully nonlinear equation of Monge-Amp\`ere type.
Given $\psi \in C^{\infty} (\bM)$ and
$\varphi \in C^{\infty} (\partial M)$ we seek solutions
$u \in C^{\infty} (\bM)$ with $\chi_u > 0$ so that
equation~\eqref{CH-I10} is elliptic; we shall call such functions
in $C^{2} (\bM)$
{\em admissible} or {\em  strictly $\chi$-plurisubharmonic}.
So we require $\psi > 0$; when $\psi \geq 0$ equation~\eqref{CH-I10}
becomes degenerate.

Fully nonlinear elliptic and parabolic equations
have close and natural connections with problems in
complex geometry and analysis.
In K\"ahler geometry, the classical Calabi conjectures~\cite{Calabi56},
\cite{Yau78},
the K\"ahler-Ricci flow and solitons,
and Donaldson's conjectures~\cite{Donaldson99}
concerning geodesics in the space of K\"ahler metrics,
among others, all reduce to the existence, regularity and other questions
about complex Monge-Amp\`ere equations on K\"ahler manifolds.
There has also been increasing interest in fully nonlinear equations
other than the complex Monge-Amp\`ere equation on complex or symplectic
manifolds from geometric problems.
In \cite{Donaldson99a} Donaldson proposed the following equation
\begin{equation}
\label{CH-I10'}
\chi_u^n =  c \chi_u^{n-1} \wedge \omega
\end{equation}
on closed K\"ahler manifolds in connection with moment maps.
The equation was studied by Chen~\cite{Chen04},
Weinkove and Song~\cite{Weinkove04},
\cite{Weinkove06}, \cite{SW08} using parabolic methods.
More recently Fang, Lai and Ma~\cite{FLM11} extended the results of \cite{SW08}
to a class of fully nonlinear equations that covers both \eqref{CH-I10}
(where $\psi$ is constant) and \eqref{CH-I10'}.
Another example which was important in leading us to \eqref{CH-I10} is the
{\em $V$-soliton equation} introduced by La Nave and Tian~\cite{LaNave-Tian}
in their work on K\"ahler-Ricci flow on symplectic quotients.

It is well known that a key step in solving fully nonlinear elliptic
equations is to derive {\em a priori} estimates up to second
order derivatives. For the Dirichlet problems in K\"ahler or Hermitian
manifolds there arise new substantial technical difficulties in
deriving these estimates.
Our main interest in this paper is to search for general
techniques to overcome these difficulties.
The methods presented in this paper work for more general equations
and produce new results even for fully nonlinear equations in $\bfC^n$
and their real counterpart in $\bfR^n$.
We shall focus on the Dirichlet problem~\eqref{CH-I10} here; other
equations will be treated in forthcoming papers.

We first state our result on the global estimates.

\begin{theorem}
\label{gl-thm10}
Let $u \in C^{4} (M) \cap C^2 (\bM)$ be an admissible solution of
equation~\eqref{CH-I10} where $\psi \in C^2 (\bM)$, $\psi > 0$.
Suppose that there exists a function $\ul{u} \in C^2 (\bM)$ satisfying
\begin{equation}
\label{CH-I20'}
n \chi_{\ul{u}}^{n-1} > (n-1) \psi \chi_{\ul{u}} \wedge \omega^{n-2}
\;\; \mbox{on $M$}.
\end{equation}
Then there exist $C_1$, $C_2$ both depending on $|u|_{C^0(\bM)}$,
$|\ul{u}|_{C^2({M})}$, the positive lower bound of $\chi_{\ul{u}}$, and  geometric quantities of $M$ such that
 \begin{equation}
\label{CH-I30}
\max_{\bM} |\nabla u| \leq  C_1 (1 + \max_{\partial M} |\nabla u|)
\end{equation}
and
\begin{equation}
\label{CH-I30'}
\max_{\bM}  
|\Delta u| \leq C_2 (1 + \max_{\partial M} |\Delta u|).
\end{equation}
In particular, if $M$ is closed, i.e. $\partial M =\emptyset$, then
\begin{equation}
\label{CH-I30c}
|\nabla u| \leq C_1,  \;\;
|\Delta u| \leq C_2
\;\; \mbox{on $M$}.
\end{equation}
The constant $C_1$ depends in addition on $\inf_M \psi > 0$ but $C_2$
only on $|\psi^{\frac{1}{n-1}}|_{C^{1,1} (\bM)}$.
\end{theorem}

The estimate for $\Delta u$ in \eqref{CH-I30c} is due to
Fang, Lai and Ma~\cite{FLM11} who introduced the cone condition~\eqref{CH-I20'},
when both $\omega$ and $\chi$ are K\"ahler and $\psi$ is a positive constant invariant.
In our case the main issue is to control
some extra third derivative terms which occur due to the nontrivial
torsion of $(M, \omega)$ when it is only assumed to be Hermitian.
The estimate also holds in the degenerate case ($\psi \geq 0$) since 
it does not require $\psi$ to have a positive lower bound, and is 
independent of the gradient bound, i.e. $C_2$ is independent of $C_1$ 
in \eqref{CH-I30}.
The gradient estimate in \eqref{CH-I30} which is crucial in order
to solve the Dirichlet problem~\eqref{CH-I10}, is new even when
$(M, \omega)$ is K\"ahler and $\chi = \omega$.

Our second primary goal in this paper is to derive boundary estimates
for second order derivatives on general Hermitian manifolds with arbitrary 
boundary (without further restrictions except being smooth and compact). 
We are able to do this under the assumption of existence of a subsolution,
and consequently solve the Dirichlet problem~\eqref{CH-I10}.

\begin{theorem}
\label{gl-thm20}
Let $(M^n, \omega)$ be a compact Hermitian manifold with smooth boundary
$\partial M$, $\varphi \in C^{\infty} (\partial M)$, $\psi \in C^{\infty} (\bM)$
and $\psi > 0$.
Suppose there exists an admissible subsolution $\ul{u} \in C^2 (\bM)$ satisfying
 \begin{equation}
\label{CH-I20}
\chi_{\ul{u}}^n \geq  \psi \chi_{\ul{u}} \wedge \omega^{n-1}
\;\; \mbox{on $\bM$}
\end{equation}
and $\ul{u} = \varphi$ on $\partial M$.
Then the Dirichlet problem~\eqref{CH-I10} admits a unique
admissible solution $u \in C^{\infty} (\bM)$. Moreover,
\begin{equation}
\label{gl3-C2}
|u|_{C^2 (\bM)} \leq C.
\end{equation}
where $C$ depends on $\ul{u}$ up to its second derivatives and
other known data.
\end{theorem}

Theorem~\ref{gl-thm20} was proved for the complex Monge-Amp\`ere equation by
the authors in \cite{GL10}.
While equation~\eqref{CH-I10} resembles the complex Monge-Amp\`ere equation
in many ways, there are some crucial differences and one needs
substantially new techniques to derive the desired estimates.
Especially, Theorem~\ref{gl-thm20} is new even when $M$ is a
bounded domain in $\bfC^n$ with $\chi = 0$. In \cite{LiSY04} Li
considered the Dirichlet problems for complex Hessian equations in
$\bfC^n$ assuming the existence of a {\em strict}
subsolution. Equation~\eqref{CH-I10}, however, fails to satisfy some
of the key structure conditions in \cite{LiSY04}.

The cone condition~\eqref{CH-I20'} is weaker than the subsolution assumption~\eqref{CH-I20}. On a closed manifold,
the existence of an admissible subsolution implies that either it is a solution or the equation does not admit a solution;
see Section~\ref{glq-P}.

Fundamental existence theorems
were established by Yau~\cite{Yau78} for the complex Monge-Amp\`ere equation
on closed K\"ahler manifolds (see also \cite{Aubin78}),
and by Caffarelli, Kohn, Nirenberg and Spruck~\cite{CKNS} for
the Dirichlet problem on strongly pseudoconvex domains in $\bfC^n$.
Cherrier and Hanani~\cite{Cherrier87}, 
\cite{CH99}, \cite{Hanani96a} made the
first effort to extend these results to Hermitian manifolds, followed by
\cite{GL10}, \cite{TWv10a}, \cite{TWv10b}, \cite{Zhang10} etc. In
\cite{TWv10b} Tosatti and Weinkove were able to extend Yau's
$C^0$ estimate completely to closed Hermitian manifolds
 and therefore established Calabi-Yau theorem for the Bott-Chern cohomology
 class.
It would be desirable to derive $C^0$ estimate for equation~\eqref{CH-I10}
on closed Hermitian manifolds.
In a different direction, the main result of \cite{CKNS}
was extended by the first author~\cite{Guan98b}
to general domains in $\bfC^n$ under the assumption of existence of a
subsolution, which found useful applications in some
important work; see e.g. ~\cite{GuanPF02}, \cite{GuanPF08}, Chen~\cite{Chen00}, Blocki~\cite{Blocki}, and Phong and Sturm~\cite{PS}.
 This is part of the motivation for us not to impose
any further assumption on the boundary beyond smoothness in
Theorem~\ref{gl-thm20}, so that it would be more convenient to use in
applications.

The parabolic Monge-Amp\`ere equation on closed Hermitian manifolds was
studied by Gill~\cite{Gill11}, Tosatti and Weinkove~\cite{TWv}.
In a series of papers~\cite{ST10}, \cite{ST11a}, \cite{ST11b}, \cite{ST12}
Streets and Tian investigated general curvature flows of Hermitian metrics.
Generalizations of Yau's theorems to symplectic four-manifolds with
compatible almost complex structures were proposed by
Donaldson~\cite{Donaldson06}, and studied by
Weinkove~\cite{Weinkove07}, Tosatti, Weinkove and Yau~\cite{TWY08},
\cite{TWv11a}, \cite{TWv11b}.

Deriving gradient estimates for fully nonlinear equations on complex
manifolds turns out a rather surprisingly difficult task. For
the Monge-Amp\`ere equation it was carried out by Cherrier~\cite{Cherrier87}
and later rediscovered by P.-F. Guan~\cite{GuanPF} and Blocki~\cite{Blocki09}.
It was, however, only very recently that Dinew and Kolodziej~\cite{DK} were
able to prove the gradient estimate for complex Hessian equations on
K\"ahler manifolds using scaling
techniques and Liouville type theorems. It would interesting to extend our
proof of the gradient estimate to more
general equations, at least for $\chi$-plurisubharmonic solutions.

\begin{remark}
Theorem~\ref{gl-thm20} was proved in \cite{GL10} for the complex Monge-Amp\`ere
equation 
under the stronger assumption $\ul{u} \in C^4 (\bM)$; see Theorem 1.1 of
\cite{GL10}. Using the methods of this paper it can be weakened
to $\ul{u} \in C^2 (\bM)$.
\end{remark}

The rest of the paper is organized as follows.
In Section~\ref{glq-P} we recall some basic formulas on Hermitian
manifolds and present a crucial lemma on which our estimates in the subsequent
sections will heavily depend. In Section~\ref{glq-C1} and
Section~\ref{glq-C2} we derive the maximum principle for $|\nabla u|$ and
$\Delta u$, the gradient and Laplacian of $u$, respectively, completing
the proof of Theorem~\ref{gl-thm10}.
Section~\ref{gblq-B} is devoted to the boundary estimates for second
derivatives. Once these estimates are established, equation~\eqref{CH-I10}
becomes uniformly elliptic. We can therefore come back to derive global
estimates for all (real) second derivatives as in Section 5 in \cite{GL10}
and apply the Evans-Krylov theorem for higher derivative estimates;
Theorem~\ref{gl-thm20} then may be proved by the continuity method.
We shall omit these standard steps.

We wish to thank Wei Sun for carefully reading the manuscript and
pointing out some errors in the previous versions.

\bigskip
\section{Preliminaries}
\label{glq-P}
\setcounter{equation}{0}
\medskip

In this section we briefly recall some basic formulas to be used in the
subsequent sections. We shall mostly follow notations in \cite{GL10}.
Throughout the paper, $g$ and $\nabla$ will denote the Riemannian metric
and Chern connection of $(M, \omega)$.
The torsion and curvature of $\nabla$ are defined respectively by
\begin{equation}
\label{cma-K95}
\begin{cases}
   T (X, Y)  = \nabla_X Y - \nabla_Y X - [X,Y], \\
 R (X, Y) Z  = \nabla_X \nabla_Y Z - \nabla_Y \nabla_X Z - \nabla_{[X,Y]} Z.
\end{cases}
\end{equation}

In local coordinates $z = (z_1, \ldots, z_n)$, $z_j =  x_j + \sqrt{-1} y_j$
we denote
$\partial_j = \partial/\partial z_j$, $\bar{\partial}_j = \partial/\partial \bz_j$,
$1 \leq j \leq n$,
and use the following notation
\begin{equation}
\label{cma-K70}
\left\{ \begin{aligned}
 g_{i \bj} \,& = g (\partial_i, \bar{\partial}_j),
                  \;\; \{g^{i\bj}\} = \{g_{i\bj}\}^{-1} , \\
 T_{ij\bk} \,& = g (T (\partial_i, {\partial}_j), \bar{\partial}_k), \;\;
      T^k_{ij} = g^{k\bl} T_{ij\bl},  \\
 R_{i\bj k\bl} \,& = g(R (\partial_i, \bar{\partial}_j) \partial_k, \bar{\partial}_l)
\end{aligned} \right.
\end{equation}
and, for a function $v \in C^4 (M)$,
$v_{i\bj} = v_{\bj i} = \partial_i \bpartial_j v$,
$v_{i\bj k} = \partial_k v_{i\bj} - \Gamma_{ki}^l v_{l\bj}$ and
\[ v_{i\bj k\bl}
   = \bpartial_l v_{i\bj k} - \ol{\Gamma_{lj}^q} v_{i\bq k}. \]
Recall that $v_{i\bj} - v_{\bj i}  = 0$, $v_{ij} - v_{ji}  = T_{ij}^l v_l$ and
(see  e.g, \cite{GL10} and \cite{GL12})
\begin{equation}
\label{gblq-B147}
\left\{ \begin{aligned}
 v_{i \bj k} - v_{k \bj i} = \,& T_{ik}^l v_{l\bj}, \;\;
v_{i \bj k} - v_{i k \bj} = - g^{l\bm} R_{k \bj i \bm} v_l, \\
v_{i\bj k\bl} - v_{i\bj \bl k}
      = \,& g^{p\bq} R_{k\bl i\bq} v_{p\bj}
          - g^{p\bq} R_{p \bl k \bj} v_{i\bq}, \\
v_{i \bj k \bl} - v_{k \bl i \bj}
  = \,&  g^{p\bq} (R_{k\bl i\bq} v_{p\bj} - R_{i\bj k\bq} v_{p\bl})
        + T_{ik}^p v_{p\bj \bl} + \ol{T_{jl}^q} v_{i\bq k}
        - T_{ik}^p \ol{T_{jl}^q} v_{p\bq}.
 \end{aligned}  \right.
\end{equation}

Let $u \in C^4 (M)$ be a solution of equation~\eqref{CH-I10}.
We denote
\[ \fg_{i\bj} = \chi_{i\bj} + u_{i\bj}, \;
   \{\fg^{i\bj}\} = \{\fg_{i\bj}\}^{-1}, \;
   W = \tr \chi + \Delta u, \;
   F^{i\bj} = \fg^{i\bj} - \frac{\delta_{i j}}{W}. \]
Locally equation~\eqref{CH-I10} then takes the form
\begin{equation}
\label{cma2-M10}
 \det \fg_{i\bj} = \frac{\psi}{n} W \det g_{i\bj}.
\end{equation}

Let $\ul{u} \in C^2 (M)$ and $\chi_{\ul{u}} > 0$ on $M$. Then
\begin{equation}
\label{gblq-C214}
 \epsilon \omega \leq \chi_{\ul{u}} \leq \epsilon^{-1} \omega
 \end{equation}
for some $\epsilon > 0$ and therefore
\begin{equation}
\label{gblq-C215}
F^{i\bj} (\chi_{i\bj} + \ul{u}_{i\bj}) \geq \epsilon F^{i\bj} g_{i\bj}.
\end{equation}

It is well known that $A \rightarrow (\det A/\tr A)^{\frac{1}{n-1}}$ is
concave for positive definite Hermitian matrices $A$ and hence so is
$A \rightarrow \log \det A - \log \tr A$. It follows that
\[ F^{i\bj} (\ul{u}_{i\bj} - u_{i\bj}) \geq 0 \;\; \mbox{in $M$}. \]
if $\ul{u}$ is a subsolution.
As a consequence we see that $u -\ul{u}$ is constant if $M$ is a closed
manifold. The following {\em strict concavity} property plays crucial roles
in our estimates.

\begin{lemma}
\label{gblq-lemma-C20}
Let $\ul{u} \in C^2 (M)$,
$\chi_{\ul{u}} > 0$ and satisfy ~\eqref{CH-I20'}.
There exist uniform constants $\theta, N > 0$ depending on $\epsilon$
and $\sup_M \psi$ such that when $W \geq N$,
\begin{equation}
\label{gblq-C210}
F^{i\bj} (\ul{u}_{i\bj} - u_{i\bj}) \geq \theta (1 + F^{i\bj} g_{i\bj}).
\end{equation}
\end{lemma}

This lemma was proved in \cite{FLM11} in a slightly different form for
a class of equations including \eqref{CH-I10}.
We include a proof which seems simpler in our special case, under the stronger assumption~\eqref{CH-I20}.

\begin{proof}[Proof of Lemma~\ref{gblq-lemma-C20}]
At a fixed point assume that $g_{i\bj} = \delta_{ij}$ and
$\{\fg_{i\bj}\}$ is diagonal.
By \eqref{gblq-C215} if some $\fg_{i\bi}$ is very small then 
\eqref{gblq-C210} clearly holds. So we may assume
$\fg_{1\bar{1}} \geq \dots \geq \fg_{n\bn} \geq c_1$
(and therefore, $n \fg_{1\bar{1}} \geq W$)
where $c_1 > 0$ depends on $\epsilon$.
We have
\begin{equation}
\label{gblq-C220}
  \begin{aligned}
 \Big(\sum_{i \geq 2} \fg^{i\bi} (\chi_{i\bi} + \ul{u}_{i\bi})\Big)^{n-1}
 \geq \,& (n-1)^{n-1} \prod_{i \geq 2} \fg^{i\bi}  (\chi_{i\bi}
          + \ul{u}_{i\bi}) \\
 \geq \,& (n-1)^{n-1}\frac{\fg_{1\bar{1}} \det (\chi_{i\bj} + \ul{u}_{i\bj})}
          {(\chi_{1\bar{1}} + \ul{u}_{1\bar{1}}) \det (\fg_{i\bj})} \\
 \geq \,& (n-1)^{n-1} \frac{(\tr \chi + \Delta \ul{u}) \fg_{1\bar{1}}}
          {(\chi_{1\bar{1}} + \ul{u}_{1\bar{1}}) W}.
 \end{aligned}
 \end{equation}
 The second inequality in \eqref{gblq-C220} follows from
\[ \prod (\chi_{i\bi} + \ul{u}_{i\bi}) \geq \det (\chi_{i\bj}
   + \ul{u}_{i\bj}) \]
While the third from the fact that $\ul{u}$ is a subsolution and $u$ a
solution of equation~\eqref{cma2-M10}.

Assume now that $W \geq N$. It follows from equation~\eqref{cma2-M10} that
for any $i = 2, \dots n$,
\begin{equation}
\label{gblq-C230}
 W = \frac{n}{\psi} \fg_{1\bar{1}} \cdots \fg_{n\bn}
     \geq \frac{n}{\psi} \fg_{1\bar{1}} \fg_{i\bi} c_1^{n-2}
     \geq \frac{N c_1^{n-2} \fg_{i\bi}}{\sup \psi}  \equiv A \fg_{i\bi}.
\end{equation}
Therefore,
\begin{equation}
\label{gblq-C240}
 \frac{\fg_{1\bar{1}}}{W} = 1 - \frac{1}{W} \sum_{i \geq 2} \fg_{i\bi}
          \geq 1 - \frac{n-1}{A}.
\end{equation}
On the other hand, by \eqref{gblq-C214},
\[ \frac{\tr \chi + \Delta \ul{u}}{\chi_{1\bar{1}} + \ul{u}_{1\bar{1}}}
      = 1 + \frac{1}{\chi_{1\bar{1}}
        + \ul{u}_{1\bar{1}}} \sum_{i \geq 2}(\chi_{i\bi} + \ul{u}_{i\bi})
   \geq 1 + (n-1) \epsilon^2. \]
Consequently,
\[ \sum \fg^{i\bi} (\chi_{i\bi} + \ul{u}_{i\bi})
   \geq (n-1) (1+ (n-1) \epsilon^2))^{\frac{1}{n-1}}
        (1-(n-1)A^{-1})^{\frac{1}{n-1}}. \]
Note that $F^{i\bi} = \fg^{i\bi} - \frac{1}{W}$,
$F^{i\bi} \fg_{i\bi} = n-1$.
By \eqref{gblq-C215}, fixing $N$ large we derive \eqref{gblq-C210}.
\end{proof}

\section{Gradient estimates}
\label{glq-C1}
\setcounter{equation}{0}
\medskip

In this section we derive \eqref{CH-I30} under assumption~\eqref{CH-I20'}. For this purpose we consider $\phi =A e^{B \eta}$ where $\eta = \ul{u} - u - \inf_M (\ul{u} - u)$
and $A, B$ are positive constants to be determined later.
Suppose the function $e^{\phi} |\nabla u|^2$ attains its maximum at
an interior point $p \in M$ (otherwise we are done). We choose local coordinates around $p$ such that
 $g_{i\bj} = \delta_{ij}$ and $\fg_{i\bj}$ is diagonal at $p$ where, unless
otherwise indicated, the computations below are evaluated.



 For each $i = 1, \ldots, n$, we have
\begin{equation}
\label{gblq-G30}
  \frac{(|\nabla u|^2)_i}{|\nabla u|^2} + \phi_i = 0, \;\;
\frac{(|\nabla u|^2)_{\bi}}{|\nabla u|^2} + \phi_{\bi} = 0
\end{equation}
and
\begin{equation}
\label{gblq-G40}
 \frac{(|\nabla u|^2)_{i\bi}}{|\nabla u|^2}
-  \frac{|(|\nabla u|^2)_{i}|^2}{|\nabla u|^4}
  + \phi_{i\bi} \leq 0.
\end{equation}

A straightforward calculation shows that
\begin{equation}
\label{cma2-M50}
 (|\nabla u|^2)_i =  u_k u_{i\bk} + u_{ki} u_{\bk},
\end{equation}
\begin{equation}
\label{cma2-M60}
\begin{aligned}
(|\nabla u|^2)_{i\bi}
= \, & u_{k \bi} u_{i \bk} + u_{ki} u_{\bk \bi}
         + u_{ki\bi} u_{\bk} + u_k u_{i\bk \bi} \\
 = \, & u_{ki} u_{\bk \bi}+ u_{i\bi k} u_{\bk} + u_{i \bi \bk} u_k
        + R_{i \bi k \bl} u_l u_{\bk} \\
      & +  \sum_k |u_{i\bk} - T^k_{il} u_{\bl}|^2 - \sum_k |T^k_{il}  u_{\bl}|^2.
\end{aligned}
\end{equation}
On the other hand, differentiating equation~\eqref{cma2-M10} gives
\[ F^{i\bi} u_{i\bi k} = f_k - F^{i\bi} \chi_{i\bi k} \]
where $f = \log \psi$.
It follows that
\begin{equation}
\label{bglq-M80}
F^{i\bi} (|\nabla u|^2)_{i\bi}
  \geq F^{i\bi} u_{ki} u_{\bk \bi}
   +   \sum_k F^{i\bi}|u_{i\bk} - T^k_{il} u_{\bl}|^2
   - C |\nabla u|^2 \sum F^{i\bi} - 2 |\nabla u| |\nabla f|.
\end{equation}

By \eqref{gblq-G30} and \eqref{cma2-M50},
\begin{equation}
\label{gblq-G50}
\begin{aligned}
 |(|\nabla u|^2)_i|^2 = \,& |u_{\bk} u_{ki}|^2
 - 2 |\nabla u|^2  \fRe \{u_k u_{i\bk} \phi_{\bi}\} - |u_k u_{i\bk}|^2 \\
     \end{aligned}
\end{equation}
Combining \eqref{gblq-G40}, \eqref{gblq-G50} and \eqref{bglq-M80}, we obtain
\begin{equation}
\label{cma2-M40}
\begin{aligned}
 |\nabla u|^2 F^{i\bi} \phi_{i\bi} +  2 F^{i\bi} \fRe \{u_k u_{i\bk} \phi_{\bi}\}
      \leq 2 |\nabla u| |\nabla f| + C |\nabla u|^2 \sum F^{i\bi}.
  \end{aligned}
\end{equation}

Next,
\[ \phi_i = B \phi \eta_i, \;\; \phi_{i\bi}
    = B \phi (B \eta_i \eta_{\bi} + \eta_{i\bi}). \]
We have
\begin{equation}
\label{gblq-G60'}
\begin{aligned}
 2 \phi^{-1} F^{i\bi} \fRe\{u_k u_{i\bk} \phi_{\bi}\}
     = \,& 2 B F^{i\bi} \fRe\{u_k u_{i\bk} \eta_{\bi}\} \\
     = \,& 2 B F^{i\bi} \fRe\{\fg_{i\bi} u_i \eta_{\bi} - u_k \chi_{i\bk} \eta_{\bi}\} \\
  \geq \,& 2 B F^{i\bi} \fg_{i\bi} \fRe\{u_i \eta_{\bi}\}
           - \frac{B^2}{2} |\nabla u|^2 F^{i\bi} \eta_{i} \eta_{\bi}
           - C \sum F^{i\bi}.
     \end{aligned}
\end{equation}
Therefore, by \eqref{cma2-M40},
\begin{equation}
\label{gblq-G70'}
\begin{aligned}
\frac{B}{2} F^{i\bi} \eta_{i} \eta_{\bi}
 + F^{i\bi} \eta_{i\bi} 
 + \,& 2 F^{i\bi} \fg_{i\bi} \frac{\fRe\{u_i \eta_{\bi}\}}{|\nabla u|^{2}} \\
 \leq \,& \frac{2 |\nabla f|}{B \phi |\nabla u|} +
              C \Big(\frac{1}{B |\nabla u|^{2}} + \frac{1}{B \phi}\Big)
 \sum F^{i\bi}. 
      \end{aligned}
\end{equation}

Suppose now that $W \geq N$ where $N$ is sufficiently large so that
\eqref{gblq-C210} holds. We consider two cases:
(a) $\fg_{j\bj} < c_1$ for some $j$ and
(b) $\fg_{i\bi} \geq c_1$ for all $1 \leq i \leq n$
where $c_1 > 0$ is a (sufficiently small) constant to be determined.

In case (a) we have by Lemma~\ref{gblq-lemma-C20}
\[ F^{i\bi} \eta_{i\bi} \geq \frac{\theta}{2} F^{j\bj}
   + \frac{\theta}{2} \Big(1 + \sum F^{i\bi}\Big)
\geq \frac{\theta}{2} \Big(\frac{1}{c_1} - \frac{1}{N}\Big)
      + \frac{\theta}{2} \Big(1 + \sum F^{i\bi}\Big). \]
Note that
\[  F^{i\bi} \fg_{i\bi} \frac{\fRe\{u_i \eta_{\bi}\}}{|\nabla u|^{2}}
 \leq \frac{|\nabla \eta|}{|\nabla u|} \leq 2 \]
provided that $|\nabla u| \geq |\nabla \ul{u}|$.
So if $c_1$ is chosen sufficiently small,
we derive a bound $|\nabla u| \leq C$ by \eqref{gblq-G70'} when
$B$ is sufficiently large.

In case (b) assume that $\fg_{1\bar{1}} \geq \dots \geq \fg_{n\bn}$.
Then as in \eqref{gblq-C230} we have for all $i \geq 2$
\[  W = \frac{n}{\psi} \fg_{1\bar{1}} \cdots \fg_{n\bn}
     \geq \frac{n}{\psi} \fg_{1\bar{1}} \fg_{i\bi} c_1^{n-2}
     \geq \frac{W c_1^{n-2} \fg_{i\bi}}{\psi}. \]
It follows that
\begin{equation}
\label{gblq-G100}
\fg_{i\bi} \leq  c_1^{2-n} \sup_M \psi \equiv C_1, \;\; 2 \leq i \leq n,
\end{equation}
and
\begin{equation}
\label{gblq-G110}
 2 \sum_{i \geq 2} F^{i\bi} \fg_{i\bi} \frac{\fRe\{u_i \eta_{\bi}\}}{|\nabla u|^{2}}
\geq - \frac{B}{4} \sum_{i \geq 2} F^{i\bi} \eta_{i} \eta_{\bi}
  - \frac{4 C_1^2}{B} \sum_{i \geq 2} F^{i\bi}.
\end{equation}
On the other hand, by \eqref{gblq-G100} we have $W \leq \fg_{1\bar{1}} + (n-1) C_1$
and therefore,
\begin{equation}
\label{gblq-G120}
 F^{1\bar{1}} \fg_{1\bar{1}} = 1 - \frac{\fg_{1\bar{1}}}{W} \leq \frac{(n-1) C_1}{N}.
\end{equation}

Plug \eqref{gblq-G110} and \eqref{gblq-G120} into \eqref{gblq-G70'}
and assume $|\nabla u| \geq |\nabla \ul{u}|$. we obtain
\begin{equation}
\label{gblq-G130}
 F^{i\bi} \eta_{i\bi} \leq
    \frac{4(n-1) C_1}{N} + \frac{2 |\nabla f|}{B \phi |\nabla u|}
     + \Big(\frac{4 C_1^2}{B} + \frac{C}{|\nabla u|^{2}}
     + \frac{C}{B \phi}\Big) \sum F^{i\bi}.
\end{equation}
By Lemma~\ref{gblq-lemma-C20}, this gives a bound $|\nabla u| \leq C$ if we fix
$N$ and $B$ sufficiently large.

Suppose now that $W \leq N$. By equation~\eqref{cma2-M10} we see that
$\fg_{i\bi} \geq c_1$, $1 \leq i \leq n$
for some $c_1 > 0$ depending on $\inf_M \psi$. It follows that
\begin{equation}
\label{gblq-G70}
  \frac{1}{c_1} \geq F^{i\bi} = \fg^{i\bi} - \frac{1}{W}
           = \frac{\fg^{i\bi}}{W} \sum_{j\neq i} \fg_{j\bj}
        \geq \frac{(n-1) c_1}{\fg_{i\bi} N}
        \geq \frac{(n-1) c_1}{N^2} \equiv c_0.
\end{equation}
Thus $F^{i\bi} \eta_{i} \eta_{\bi} \geq c_0 |\nabla \eta|^2$.
Plugging these back in \eqref{gblq-G70'}
we derive $|\nabla \eta| \leq C$ which in turn implies a bound
 $|\nabla u| \leq C$.
This completes the proof of \eqref{CH-I30}.

\bigskip

\section{The second order estimates}
\label{glq-C2}
\setcounter{equation}{0}
\medskip

In this section we derive the second order estimates \eqref{CH-I30'}.
As in the previous section we assume $\ul{u} \in C^2 (\bM)$ satisfies
\eqref{gblq-C214}  and \eqref{CH-I20'}.


\begin{proposition}
\label{gblq-prop-C10}
Let $u \in C^4 (M)$ be a solution of equation~\eqref{CH-I10}
and $W = \Delta u + \tr \chi$.
Suppose $\psi \geq 0$ and $\psi^{\frac{1}{n-1}} \in C^{1,1} (\bM)$.
Then there exists $C > 0$ depending on $\max_{\bM}$, $\epsilon$ in
\eqref{gblq-C214} and
the $C^{1,1}$ norms of $\chi$, $\ul{u}$ and $\psi^{\frac{1}{n-1}}$
as well as the geometric quantities of $M$
such that
\begin{equation}
\label{gblq-I60}
W \leq C (1 + \max_{\partial M} W) \;\; \mbox{on $M$}.
\end{equation}
\end{proposition}

\begin{proof}
We assume $\psi > 0$; the general case $\psi \geq 0$ follows from
approximation.
Let $\varPhi = e^\phi W$ where $\phi$
is a function to be determined.
Suppose $\varPhi$ achieves its maximum at an interior point $p \in M$.
Choose local coordinates around $p$  such that $g_{i\bj} = \delta_{ij}$
and $\fg_{i\bj}$ is diagonal at $p$.
 we have (all calculations  below are done at $p$),
\begin{equation}
\label{gblq-C80}
  \frac{W_i}{W} + \phi_i = 0, \;\;
\frac{W_{\bi}}{W} + \phi_{\bi} = 0,
\end{equation}
\begin{equation}
\label{gblq-C90}
 \frac{W_{i\bi}}{W}
-  \frac{|W_{i}|^2}{W^2}
  + \phi_{i\bi} \leq 0.
\end{equation}

Now,
\begin{equation}
\label{gblq-C905a}
  \begin{aligned}
 |W_i|^2 = \,& \Big|\sum_j \fg_{j\bj i}\Big|^2
         =    \Big|\sum_j (\fg_{i\bj j} - T_{ij}^j \fg_{j\bj}) + \lambda_i\Big|^2 \\
         = \,& \Big|\sum_j (\fg_{i\bj j} - T_{ij}^j \fg_{j\bj})\Big|^2
             + 2 \sum_j \fRe\{(\fg_{i\bj j} - T_{ij}^j \fg_{j\bj}) \ol{\lambda_i}\}
               + |\lambda_i|^2
         \end{aligned}
\end{equation}
where
\[ \lambda_{i} = \sum_j (\chi_{j\bj i} - \chi_{i\bj j} + T_{ij}^l \chi_{l\bj}). \]
By Schwarz inequality,
\begin{equation}
\label{gblq-C905b}
\Big|\sum_j (\fg_{i\bj j} - T_{ij}^j \fg_{j\bj})\Big|^2
\leq W \sum_j \fg^{j\bj} |\fg_{i\bj j} - T_{ij}^j \fg_{j\bj}|^2.
\end{equation}

By \eqref{gblq-C80} we write
\[ 2 \sum_j \fRe\{(\fg_{i\bj j} - T_{ij}^j \fg_{j\bj}) \ol{\lambda_i}\}
   = 2 \fRe\{(W_i - \lambda_i) \ol{\lambda_i}\}
   = - 2 W \fRe\{\phi_i  \ol{\lambda_i}\} - 2 |\lambda_i|^2. \]
From 
\eqref{gblq-C905a} and \eqref{gblq-C905b} we see that
\begin{equation}
\label{gblq-C905}
  \frac{|W_i|^2}{W} \leq \fg^{j\bj} |\fg_{i\bj j} - T_{ij}^j \fg_{j\bj}|^2
         - 2 \fRe\{\phi_i  \ol{\lambda_i}\}.
 \end{equation}

Differentiating equation~\eqref{cma2-M10} twice we obtain (at $p$)
\begin{equation}
\label{gblq-C75}
\fg^{i\bi} \fg_{i\bi k\bk} - \fg^{i\bi} \fg^{j\bj} \fg_{i\bj k} \fg_{j\bi \bk}
- \frac{W_{k\bk}}{W} +
 \frac{|W_k|^2}{W^2}= f_{k\bk}.
\end{equation}
From \eqref{gblq-B147}, 
\begin{equation}
\label{gblq-R155}
 \begin{aligned}
\fg_{i \bi k \bk} - \fg_{k \bk i \bi}
   = \,& R_{k\bk i\bi} \fg_{i\bi} - R_{i\bi k\bk} \fg_{k\bk}
         + 2 \fRe\{\ol{T_{ik}^j} \fg_{i\bj k}\}
         - |T_{ik}^j|^2 \fg_{j\bj} - G_{i\bi k\bk}
  \end{aligned}
 \end{equation}
where $G_{i\bi k\bk} = \chi_{k \bk i \bi} - \chi_{i \bi k \bk}
     +  R_{k\bk i\bp} \chi_{p\bi} - R_{i\bi k\bp} \chi_{p\bk}
     + 2 \fRe\{\ol{T_{ik}^j} \chi_{i\bj k}\}
     - T_{ik}^p \ol{T_{ik}^q} \chi_{p\bq}$.
Combining \eqref{gblq-C75} and \eqref{gblq-R155} gives
\begin{equation}
\label{gblq-C120}
\begin{aligned}
F^{i\bi} W_{i\bi}
   = 
     \,& \sum_{i,j,k} \fg^{i\bi} \fg^{j\bj} |\fg_{i\bj k} - T_{ik}^j  \fg_{j\bj}|^2
         - \frac{|\nabla W|^2}{W^2} \\
       & \;\;\;\;\;\;\;\;\;\;\;\;
          + \Delta f + (\fg^{i\bi} \fg_{k\bk} - 1) R_{i\bi k\bk}
         + \fg^{i\bi} G_{i\bi k\bk} \\
\geq \,& \fg^{i\bi} \fg^{j\bj} |\fg_{i\bj j} - T_{ij}^j  \fg_{j\bj}|^2
         - \frac{|\nabla W|^2}{W^2}
        + \Delta f - (C_1 W + C_2) \sum \fg^{i\bi} - C_3
       \end{aligned}
\end{equation}
where
\[ C_1 = - \inf_M \inf_{i,j} R_{i\bi j\bj}, \;\; 
   C_2 = \inf_M \inf_i \sum_k G_{i\bi k\bk}, \;\;
   C_3 = \sup_M \sum_{i,j} R_{i\bi j\bj}. \]
By \eqref{gblq-C90}, \eqref{gblq-C905} and \eqref{gblq-C120},
\begin{equation}
\label{gblq-C150}
\begin{aligned}
0 \geq \,& F^{i\bi} W_{i\bi} - \frac{\fg^{i\bi} |W_i|^2}{W}
           + \frac{|\nabla W|^2}{W^2} + W F^{i\bi} \phi_{i\bi} \\
  \geq \,& W F^{i\bi} \phi_{i\bi} + 2 \fg^{i\bi} \fRe\{\phi_i \ol{\lambda_i}\}
           - C W\sum \fg^{i\bi} + \Delta f.
\end{aligned}
\end{equation}

Let $\phi = e^{A \eta}$ where $\eta = \ul{u} - u + \sup_M (u - \ul{u})$ and
$A$ is a positive constant. We see that $\phi_i = A \phi \eta_i$ and
$\phi_{i\bi} =  A \phi \eta_{i\bi} + A^2 \phi \eta_i \eta_{\bi}$. Therefore,
applying Schwarz inequality,
\begin{equation}
\label{gblq-C200a}
  \begin{aligned}
2 \fg^{i\bi} \fRe\{\phi_i \ol{\lambda_i}\}
     = \,& 2 A \phi F^{i\bi} \fRe\{\eta_i \ol{\lambda_i}\}
           + 2 A \phi W^{-1} \fRe\{\eta_i \ol{\lambda_i}\} \\
 \geq \,& - A^2 \phi F^{i\bi} |\eta_i|^2 - C \phi \sum F^{i\bi}
          - \frac{C \phi}{W^2} \sum \frac{1}{F^{i\bi}} \\
 \geq \,& - W A^2 \phi F^{i\bi} |\eta_i|^2 - C \phi \Big(\psi^{\frac{1}{1-n}} + \sum F^{i\bi}\Big).
\end{aligned}
\end{equation}
The last inequality in \eqref{gblq-C200a} follows from
\[ \begin{aligned}
 F^{j\bj}
    = \,& \fg^{j\bj} - \frac{1}{W}
    =     \frac{\fg^{j\bj}}{W} \sum_{k \neq j} \fg_{k\bk}
 \geq     \frac{\fg^{j\bj}}{W}
              \Big(\prod_{k \neq j} \fg_{k\bk}\Big)^{\frac{1}{n-1}} \\
    = \,&  \frac{1}{W} (\fg^{j\bj})^{\frac{n}{n-1}} (\det \fg_{k\bl})^{\frac{1}{n-1}}
    =     \frac{1}{W^{\frac{n-2}{n-1}}} (\fg^{j\bj})^{\frac{n}{n-1}}
          \Big(\frac{\psi}{n}\Big)^{\frac{1}{n-1}}
 \geq     \frac{\psi^{\frac{1}{n-1}}}{n W^2}.
\end{aligned} \]

Suppose $W$ is sufficiently large. Then by Lemma~\ref{gblq-lemma-C20}
\begin{equation}
\label{gblq-C200}
  \begin{aligned}
\frac{1}{A \phi} F^{i\bi} \phi_{i\bi}
      = \,& F^{i\bi} \eta_{i\bi} + A F^{i\bi} \eta_i \eta_{\bi} \\
   \geq \,& A F^{i\bi} \eta_i \eta_{\bi} + \theta \Big(1 + \sum F^{i\bi}\Big).
 \end{aligned}
 \end{equation}
It follows from \eqref{gblq-C150}, \eqref{gblq-C200a} and \eqref{gblq-C200}
that
\begin{equation}
\label{gblq-C150a}
\begin{aligned}
\theta A (W - C) - C \psi^{\frac{1}{1-n}} W + \phi^{-1} \Delta f
     + W (\theta A - C - C \phi^{-1}) \sum \fg^{i\bi} \leq 0
\end{aligned}
\end{equation}
provided that $W$ is large enough. We obtain when $A$ is sufficiently large,
\begin{equation}
\label{gblq-C150b}
\frac{\theta}{2} A W \sum \fg^{i\bi} \leq
          C \psi^{\frac{1}{1-n}} W  - \Delta f.
 \end{equation}

Note that $|\Delta f| \leq C \psi^{\frac{1}{n-1}}$ where $C$ depends on
$|\psi^{\frac{1}{n-1}}|_{C^{1,1} (\bM)}$; see e.g. \cite{GTW99}.
Now assume that $\fg_{1\bar{1}} \geq \dots \geq \fg_{n\bn}$
(and therefore, $n \fg_{1\bar{1}} \geq W$).
As in \eqref{gblq-C220} we have
\begin{equation}
\label{gblq-C220'}
  \begin{aligned}
 \sum \fg^{i\bi}
 \geq \,& (n-1) \prod_{i \geq 2} (\fg^{i\bi})^{\frac{1}{n-1}}
     = \frac{(n-1) \fg_{1\bar{1}}^{\frac{1}{n-1}}}
          {\det (\fg_{i\bj})^{\frac{1}{n-1}}}
 \geq \frac{n-1}{n \psi^{\frac{1}{n-1}}}.
 \end{aligned}
 \end{equation}
Fixing $A$ sufficiently large we obtain a bound for $W$ at $p$ from \eqref{gblq-C150b}.
\end{proof}

\bigskip

\section{Boundary estimates for second derivatives}
\label{gblq-B}
\setcounter{equation}{0}
\medskip

In this section we derive
{\em a priori} estimates for second derivatives on the boundary
\begin{equation}
 \label{cma-37}
\max_{\partial M} |\nabla^2 u| \leq C.
\end{equation}

Let $u \in C^3 (\bM)$ and $\ul{u} \in C^2 (\bM)$ be an admissible
solution and subsolution of the Dirichlet problem~\eqref{CH-I10},
respectively.

\begin{proposition}
\label{gblq-prop-C10}
 Assume $\psi \in C^2 (\bM)$, $\psi > 0$
and $\varphi \in C^4  (\partial M)$ and that there exists an admissible
subsolution $\ul{u} \in C^2 (\bM)$ satisfying \eqref{CH-I20}.
Then there exists $C > 0$ depending on $\max_{\bM}$, $\epsilon$ in
\eqref{gblq-C214} and
the $C^{1,1}$ norms of $\chi$, $\ul{u}$ and $\psi^{\frac{1}{n-1}}$
as well as the geometric quantities of $M$
such that
\begin{equation}
\label{gblq-I60}
W \leq C (1 + \max_{\partial M} W) \;\; \mbox{on $M$}.
\end{equation}
\end{proposition}

As in \cite{GL10} we follow the methods of \cite{GS93}, \cite{Guan98a},
\cite{Guan98b} using subsolutions in construction of barrier functions.
The new technical tool needed in this approach is
Lemma~\ref{gblq-lemma-C20}.

To derive \eqref{cma-37} let us
consider a boundary point $0 \in \partial M$.
We use normal coordinates around $0$ such that
$x_n$ is the interior normal direction to
$\partial M$ at $0$.
For convenience we set
\[ 
t_{2k-1} = x_k,\;  t_{2k} = y_k, \, 1 \leq k \leq n-1; \;
  t_{2n-1} = y_n, \; t_{2n} = x_n\]
and write for a function $v \in C^3 (\bM)$
\[ v_{t_{\alpha}} = \partial_{t_{\alpha}} v, \;
 v_{i t_{\alpha}} = \nabla_{t_{\alpha} i} v
   := \partial_{t_{\alpha}} \partial_i v - \partial_{\nabla_{t_{\alpha}} \partial_i} v, \;
   v_{\bi t_{\alpha}} = \nabla_{t_{\alpha} \bi} v, \; \mbox{etc.} \]
We have
\begin{equation}
\label{gblq-B110}
\begin{aligned}
  v_{i x_j} -  v_{x_j i} = \,&  v_{ij} - v_{ji}  = T_{ij}^l v_l, \\
v_{i y_j} -  v_{y_j i} = \,& \sqrt{-1} (v_{ij} - v_{ji})
                           = \sqrt{-1} T_{ij}^l v_l.
\end{aligned}
\end{equation}
Similarly, $v_{\bi x_j} -  v_{x_j \bi} = \ol{T_{ij}^l} v_{\bl}$
and $v_{\bi y_j} -  v_{y_j \bi} = -  \sqrt{-1} \ol{T_{ij}^l} v_{\bl}$.
Moreover, by \eqref{gblq-B147} one derives
\begin{equation}
\label{gblq-B120}
\begin{aligned}
  v_{i \bj x_k} - v_{x_k i \bj}
    = \,& - g^{l\bm} R_{i \bj k \bm} v_l + T_{ik}^l v_{l\bj}
          + \ol{T_{jk}^l} v_{i\bl}, \\
  v_{i \bj y_k} - v_{y_k i \bj}
    = \,& \sqrt{-1} (- g^{l\bm} R_{i \bj k \bm} v_l + T_{ik}^l v_{l\bj}
                       - \ol{T_{jk}^l} v_{i\bl}).
  \end{aligned} 
\end{equation}

Since $u - \varphi = 0$ on $\partial M$, one derives
\begin{equation}
|u_{t_{\alpha} t_{\beta}}(0)| \leq C, \;\;\;\; \alpha, \beta < 2n
\label{cma-70}
\end{equation}
where $C$ depends on $|u|_{C^1 (\bM)}$, $|\ul{u}|_{C^1 (\bM)}$, and
the principal curvatures of $\partial M$.

To estimate $u_{t_{\alpha} x_n} (0)$ for $\alpha \leq 2n$,
we shall employ a barrier function of the form
\begin{equation}
\label{cma-E85}
v = (u - \ul{u}) + t \sigma - T \sigma^2
\;\; \mbox{in $\Omega_{\delta} = M \cap B_{\delta}$}
\end{equation}
where $t, T$ are positive constants to be determined,
$B_{\delta}$ is the
(geodesic) ball of radius $\delta$ centered at $p$,
and $\sigma$ is the distance function to $\partial M$.
Note that $\sigma$ is smooth in
$M_{\delta_0} := \{z \in M: \sigma (z) < \delta_0\}$
for some $\delta_0 > 0$.

\begin{lemma}
\label{cma-lemma-20}
There exists $c_0 > 0$ such that for $T$ sufficiently large and
$t, \delta$ sufficiently small,
$v \geq 0$ and
\begin{equation}
\label{eq-100}
  F^{i\bj} v_{i\bj}
    \leq - c_0 (1 + F^{i\bj} g_{i\bj})
   \;\;\; \mbox{in} \;\; \Omega_{\delta}.
\end{equation}
\end{lemma}

\begin{proof}
Note that, since $\sigma$ is smooth and $\sigma = 0$ on $\partial M$,
for fixed $t$ and $T$ we may require $\delta$ to be so small that
$v \geq 0$ in $\Omega_{\delta}$. Note also that
\[ F^{i\bj} \sigma_{i\bj} \leq C_1 F^{i\bj}  g_{i\bj} \]
for some constant $C_1 > 0$ under control. Therefore,
\begin{equation}
\label{eq-110}
  F^{i\bj} v_{i\bj} \leq F^{i\bj} (u_{i\bj} -\ul{u}_{i\bj})
    + C_1 (t + T \sigma) F^{i\bj} g_{i\bj} - 2 T F^{i\bj} \sigma_i \sigma_{\bj}
\;\; \mbox{in} \;\; \Omega_{\delta}.
\end{equation}

Fix $N > 0$ sufficiently large so that Lemma~\ref{gblq-lemma-C20}
holds. At a fixed point in $\Omega_{\delta}$, we consider two cases:
(a) $W \leq N$ and (b) $W > N$.

In case (a) let $\lambda_1 \leq \cdots \leq \lambda_n$ be
the eigenvalues of $\{\fg_{i\bj}\}$. 
We see from equation~ \eqref{cma2-M10}
that there is a uniform lower bound $\lambda_1 \geq c_1 > 0$.
Consequently,
\begin{equation}
\label{cma-E95}
F^{i\bj} \sigma_i \sigma_{\bj}
 \geq \Big(\frac{1}{\lambda_n} - \frac{1}{W}\Big) |\nabla \sigma|^2
  \geq \frac{c_1 (n-1)}{\lambda_n W} |\nabla \sigma|^2
  \geq \frac{c_1 (n-1)}{N^2} |\nabla \sigma|^2 = \frac{c_1 (n-1)}{4 N^2}.
\end{equation}

Let $T = \frac{2 n N^2}{c_1 (n-1)}$. By \eqref{eq-110} and \eqref{gblq-C210},
\begin{equation}
\label{cma-E90}
F^{i\bj} v_{i\bj} \leq n -1 + (C_1 (t + T \sigma) - \epsilon) F^{i\bj} g_{i\bj}
- \frac{c_1 (n-1)T}{2 N^2}
  \leq - \frac{\epsilon}{2} (1 + F^{i\bj} g_{i\bj})
\end{equation}
if we require $t$ and $\delta$ small enough to satisfy
$C_1 (t + T \delta) \leq \epsilon/2$.

Suppose now that $W > N$.
By Lemma~\ref{gblq-lemma-C20} and \eqref{eq-110}, we may further
require $t$ and $\delta$ to satisfy $C_1 (t + T \delta) \leq \theta/2$
so that \eqref{eq-100} holds.
\end{proof}

Using Lemma~\ref{cma-lemma-20} we may derive the estimates
$|u_{t_{\alpha} x_n} (0)| \leq C$ (and therefore
$|u_{x_n t_{\alpha}} (0)| \leq C$)
for $\alpha < 2n$  as in
\cite{GL10}; we shall only give an outline of the proof here to correct
some minor errors in \cite{GL10}.

For fixed $\alpha < 2n$, we write $\eta = \sigma_{t_{\alpha}}/\sigma_{x_n}$
and define
\[ \mathcal{T} = \nabla_{\frac{\partial}{\partial t_{\alpha}}}
     - \eta \nabla_{\frac{\partial}{\partial x_n}}. \]
We have
\[ | \mathcal{T} (u - \varphi)| + (u_{y_n} - \varphi_{y_n})^2 \leq C
    \;\; \mbox{in} \; \Omega_{\delta}. \]
On $\partial M$ since $u - \varphi = 0$ and
$\cT$ is a tangential differential operator, we have
\[  \mathcal{T} (u - \varphi) = 0 \;\; \mbox{on}\;\partial M \cap B_{\delta} \]
and, similarly,
\begin{equation}
\label{cma-175}
(u_{y_n} - \varphi_{y_n})^2 \leq C \rho^2
  \;\; \mbox{on} \; \partial M \cap B_{\delta}.
\end{equation}

In \cite{GL10} we used the formula
\begin{equation}
\fg^{i\bj}  (\mathcal{T} u)_{i\bj}
   = \fg^{i\bj}  (u_{t_\alpha i\bj} + \eta u_{x_n i\bj})
 + \fg^{i\bj} \eta_{i\bj} u_{x_n} + 2 \fg^{i\bj}  \fRe \{\eta_i u_{x_n \bj}\}.
\end{equation}
The correct one should be
\begin{equation}
\label{gblq-B160}
\fg^{i\bj}  (\mathcal{T} u)_{i\bj}
   = \fg^{i\bj}  ((u_{t_\alpha})_{i\bj} - \eta (u_{x_n})_{i\bj})
 - \fg^{i\bj} \eta_{i\bj} u_{x_n} - 2 \fg^{i\bj}  \fRe \{\eta_i (u_{x_n})_{\bj}\}.
\end{equation}
We have $(u_{t_{\alpha}})_{\bj} = u_{t_{\alpha} \bj} + u_{\nabla_j t_{\alpha}}$ and
\begin{equation}
\label{gblq-B161}
(u_{t_\alpha})_{i\bj} = u_{t_\alpha i\bj} + u_{i (\nabla_{\bj} t_{\alpha})}
   + u_{(\nabla_i t_{\alpha}) \bj} - u_{\nabla_i \nabla_j t_{\alpha}}.
 \end{equation}
Note that $u_{i (\nabla_{\bj} t_{\alpha})} + u_{(\nabla_i t_{\alpha}) \bj}$
is equal to a linear combination of terms in the form $u_{i \bl}$ and $u_{k \bj}$
since $\nabla_i \frac{\partial}{\partial \bz_k} = 0$, and that
$|F^{i\bj} \fg_{i\bk}| \leq C$.

Differentiating equation~\eqref{cma2-M10}
with respect to $t_{\alpha}$, $\alpha \leq 2n$, by \eqref{gblq-B120} we obtain
\begin{equation}
\label{gblq-B50}
\begin{aligned}
|F^{i\bj}  u_{t_{\alpha} i\bj}|
  \leq \, & |(\log \psi)_{t_{\alpha}}| + |F^{i\bj} \chi_{i\bj t_{\alpha}}|
   + |F^{i\bj} (u_{t_{\alpha} i\bj} - u_{i\bj t_{\alpha}})|
 \leq  C (1 + F^{i\bj} g_{i\bj}).
\end{aligned}
\end{equation}
Here we also used the identity
 \begin{equation}
\label{gblq-B170'}
 F^{i\bj} T^l_{ki} u_{l\bj}
 = T^i_{ki} -\frac 1 {W} T_{ki}^l \fg_{l \bj} - F^{i\bj} T^l_{ki} \chi_{l\bj}.
 \end{equation}
 Next,
\begin{equation}
\label{gblq-B190}
 \begin{aligned}
  F^{i\bj} \eta_i (u - \varphi)_{x_n \bj}
   = &\, F^{i\bj} \eta_i (2 (u - \varphi)_{n \bj}
              + \sqrt{-1} (u - \varphi)_{y_n \bj}) \\
   = &\, 2 \eta_n -2\frac{\eta_i \fg_{n \bi}}{W}
         - 2 F^{i\bj} \eta_i (\chi_{n \bj} +\varphi_{n \bj})
         + \sqrt{-1} F^{i\bj} \eta_i (u - \varphi)_{y_n \bj}.
 \end{aligned}
 \end{equation}

By \eqref{gblq-B50} and \eqref{gblq-B190} we obtain
\begin{equation}
\label{gblq-B180}
 |\fg^{i\bj}  ((u_{t_\alpha})_{i\bj} - \eta (u_{x_n})_{i\bj})|
   \leq |\mathcal{T} (f)| + C_1 (1 + \fg^{i\bj}  g_{i\bj}).
\end{equation}
Now the rest of the proof is similar to \cite{GL10}. In summary we have
\begin{equation}
\label{cma-180}
|u_{t_{\alpha} x_n} (0)| \leq C \;
 \mbox{and} \; |u_{x_n t_{\alpha}} (0)| \leq C, \;\; \alpha < 2n.
\end{equation}

To finish the proof of \eqref{cma-37}
it suffices to prove
\begin{equation}
\label{cma-200}
 \fg_{n \bn} (0) \leq C.
\end{equation}

Expanding $\det \fg_{i\bj}$, equation~\eqref{cma2-M10} takes the form
\begin{equation}
\label{gblq-B310}
 \det \fg_{i\bj}(0) = a \fg_{n\bn}(0) - b = \psi \sum  \fg_{i\bi}(0)
\end{equation}
where $a = \det (\fg_{\alpha \bar{\beta}}(0)|_{\{1 \leq \alpha, \beta \leq n-1\}})$
and $b \geq 0$ is bounded. 
To show $\fg_{n\bn}(0) \leq C$ we only have to derive a uniform
positive lower bound for $a - \psi (0)$. For this we shall use an idea of
Trudinger~\cite{Trudinger95} combined with Lemma~\ref{gblq-lemma-C20}.

Let $T_C \partial M$ 
be the complex tangent bundle of $\partial M$ and
\[ T^{1,0}_C \partial M =  T^{1,0} M \cap T_C \partial M
   = \Big\{\xi \in T^{1,0} M: d \sigma (\xi) = 0\Big\}. \]
Let $\hat{\chi_u}$ and $\hat{\omega}$ denote the restrictions to
 $T_C \partial M$ of $\chi_u$ and $\omega$ respectively. We wish to show
\[ m_0 = 
       \min_{\partial M}
     \frac{\hat{\chi_u}^{n-1}}{\psi \hat{\omega}^{n-1}} > 1. \]

Suppose $m_0$ is attained at a point $0 \in \partial M$.
Choose local coordinates around $0$ as before
such that $e_n (0)$ is normal to $\partial M$ and
$g_{i\bj} (0)= \delta_{ij}$.

Let $\zeta_1, \ldots, \zeta_{n-1}$ be a local frame of
vector fields in $T_C^{(1,0)} \partial M$ with
$g (\zeta_{\alpha}, \bar{\zeta_{\beta}}) = \delta_{\alpha \beta}$ and
$\zeta_{\alpha} (0) = \tau_{\alpha} (0)$.
Let $\fa_{\alpha \beta} = \chi_u (\zeta_{\alpha}, \bar{\zeta_{\beta}})$.
Then
\[ \hat{\chi_u}^{n-1}
   =  \det (\fa_{\alpha \bar{\beta}}) \hat{\omega}^{n-1}. \]
We extend  $\zeta_1, \ldots, \zeta_{n-1}$
 by their parallel transports
along geodesics normal to $\partial M$ so that they are smoothly defined in
a neighborhood of $0$. Similarly we assume $\varphi$ is extended smoothly
to $\bM$. We denote $\tilde{u}_{\alpha \beta} = u_{\zeta_{\alpha} \bar{\zeta_{\beta}}}$,
etc.

Define, for a positive definite $(n-1) \times (n-1)$
Hermitian matrix $\{r_{\alpha  \bar{\beta}}\}$,
\[ G[r_{\alpha \bar{\beta}}] \equiv \det (r_{\alpha \bar{\beta}})^{\frac {1}{n-1}} \]
and let
\[ G^{\alpha \bar{\beta}}_0 = \frac{\partial G}{\partial r_{\alpha \bar{\beta}}}
[\fa_{\alpha \bar{\beta}} (0)]. \]
Note that $G$ is concave and homogeneous of degree one.
Therefore, for any $\{r_{\alpha  \bar{\beta}}\} > 0$,
\[ G^{\alpha \bar{\beta}}_0 r_{\alpha  \bar{\beta}} \geq G [r_{\alpha  \bar{\beta}}]. \]

Since
$\tilde{u}_{\alpha \bar{\beta}} (0) = {u}_{\alpha \bar{\beta}} (0) =
\ul{u}_{\alpha \bar{\beta}} (0) - (u - \ul{u})_{x_n} (0) \sigma_{\alpha \bar{\beta}}$,
\[ G [\fa_{\alpha \bar{\beta}} (0)]
   = G^{\alpha \bar{\beta}}_0 \fa_{\alpha \bar{\beta}} (0)
   = G^{\alpha \bar{\beta}}_0 (\ul{u}_{\alpha \bar{\beta}} (0)
     + \chi_{\alpha \bar{\beta}} (0))
    -  (u - \ul{u})_{x_n} (0) G^{\alpha \bar{\beta}}_0 \sigma_{\alpha \bar{\beta}} (0). \]
Suppose that for some small $\theta_0 > 0$ to be determined,
\[ (u - \ul{u})_{x_n} (0) G^{\alpha \bar{\beta}}_0 \sigma_{\alpha \bar{\beta}} (0)
   \leq \theta_0  G^{\alpha \bar{\beta}}_0 (\ul{u}_{\alpha \bar{\beta}} (0)
     + \chi_{\alpha \bar{\beta}} (0)). \]
Then
\[ \begin{aligned}
  G [\fa_{\alpha \bar{\beta}} (0)]
  \geq \,& (1 - \theta_0) G^{\alpha \bar{\beta}}_0
           (\ul{u}_{\alpha \bar{\beta}} (0) + \chi_{\alpha \bar{\beta}} (0)) \\
  \geq \,& (1 - \theta_0)
           G[\ul{u}_{\alpha \bar{\beta}} (0) + \chi_{\alpha \bar{\beta}} (0)] \\
  \geq \,& (1 - \theta_0) \Big(\frac{\det (\ul{u}_{i\bj} + \chi _{i\bj})}
            {\ul{u}_{n\bn} + \chi_{n\bn}}\Big)^{\frac{1}{n-1}} \\
     = \,& (1 - \theta_0) \Big(\frac{(\Delta \ul{u} + \tr \chi) \psi}
        {\ul{u}_{n\bn} + \chi_{n\bn}}\Big)^{\frac{1}{n-1}} \\
  \geq \,& (1 - \theta_0)
           (1 + (n-1) \epsilon^2)^{\frac{1}{n-1}} \psi^{\frac{1}{n-1}} (0).
\end{aligned} \]
Choosing $\theta_0$ small enough we obtain
\[ m_0 = \frac{\det \fa_{\alpha \bar{\beta}} (0)}{\psi (0)}
         \geq 1 + \frac{\theta_0}{2}. \]

Suppose now that
\[ (u - \ul{u})_{x_n} (0) G^{\alpha \bar{\beta}}_0 \sigma_{\alpha \bar{\beta}} (0)
    > \theta_0  G^{\alpha \bar{\beta}}_0 (\ul{u}_{\alpha \bar{\beta}} (0)
      + \chi_{\alpha \bar{\beta}} (0)). \]
On $\partial M$ near $0$,
\[ \tilde{u}_{\alpha \bar{\beta}} = \tilde{\varphi}_{\alpha \bar{\beta}}
  - (u - \varphi)_{\nu}  \tilde{\sigma}_{\alpha \beta} \]
  where $\nu$  is the interior unit normal to $\partial M$.
  Write
  \[ \nu = \sum_{\alpha \leq 2n} \nu_{\alpha} \frac{\partial}{\partial t_{\alpha}}. \]
We have $|\nu_{\alpha}| \leq C \rho$ for $\alpha < 2n$ and
$|(u - \varphi)_{t_{\alpha}}| \leq C \rho$
since $\nu_{\alpha} (0) = 0$ for $\alpha < 2n$ and $u = \varphi$ on $\partial M$.
Define
\[ \begin{aligned}
  \varPhi = \,& G^{\alpha \bar{\beta}}_0 (\tilde{\varphi}_{\alpha \bar{\beta}}
                 + \tilde{\chi}_{\alpha \bar{\beta}})
                 - (u - \varphi)_{x_n} \nu_{2n} G^{\alpha \bar{\beta}}_0
                   \tilde{\sigma}_{\alpha \bar{\beta}}
                 - (m_0 \psi)^{\frac{1}{n-1}} \\
     \equiv \,& - (u - \varphi)_{x_n} \eta + Q
         \end{aligned} \]
where $\eta \equiv \nu_{2n} G^{\alpha \bar{\beta}}_0 \tilde{\sigma}_{\alpha \bar{\beta}}$
and $Q$ are smooth. On $\partial M$,
\[ \begin{aligned}
 \varPhi = \,& G^{\alpha \bar{\beta}}_0 \fa_{\alpha \bar{\beta}}
               + [(u - \varphi)_{\nu} - \nu_{2n} (u - \varphi)_{x_n}]
                  G^{\alpha \bar{\beta}}_0 \tilde{\sigma}_{\alpha \bar{\beta}}
                   - (m_0 \psi)^{\frac{1}{n-1}} \\
      \geq \,& [(u - \varphi)_{\nu} - \nu_{2n} (u - \varphi)_{x_n}] \geq - C \rho^2
     \end{aligned} \]
Note that $\varPhi (0) = 0$ and
\[ \eta (0) \geq
   \theta_0 (1 + (n-1) \epsilon^2)^{\frac{1}{n-1}}
    \psi^{\frac{1}{n-1}} (0)/(u - \ul{u})_{x_n} (0) \geq c_2 > 0. \]

We calculate as before using \eqref{gblq-B161}
\begin{equation}
\label{gblq-B360}
 \begin{aligned}
F^{i\bj} \varPhi_{i\bj}
  \leq \, & - \eta F^{i\bj} u_{x_n i\bj}
   - 2 F^{i\bj} \fRe\{\eta_i (u - \varphi)_{x_n \bj}\}
           + C F^{i\bj}  g_{i\bj}.
\end{aligned}
\end{equation}
Therefore, by \eqref{gblq-B50} and \eqref{gblq-B190}
\begin{equation}
\label{cma-105}
F^{i\bj} (\varPhi - |(u - \varphi)_{y_n}|^2)_{i\bj}
      \leq C (1 + F^{i\bj}  g_{i\bj}).
\end{equation}

Consequently, by Lemma~\ref{cma-lemma-20} we see that
$A v + B \rho^2 + \varPhi - |(u - \varphi)_{y_n}|^2 \geq 0$
on $\partial (M \cap B_{\delta} (0))$ and
\begin{equation}
\label{cma-106}
F^{i\bj} [A v + B \rho^2 + \varPhi - |(u - \varphi)_{y_n}|^2]_{i\bj}
      \leq 0
        \;\; \mbox{in $M \cap B_{\delta} (0)$}
\end{equation}
when $A \gg B \gg 1$. By the maximum principle,
$A v + B \rho + \varPhi - |(u - \varphi)_{y_n}|^2 \geq 0$
in $M \cap B_{\delta} (0)$.
Note that
$\varPhi (0) = G [\fg_{\alpha \bar{\beta}} (0)] - (m_0 \psi (0))^{\frac{1}{n-1}} = 0$.
We see that $\varPhi_{e_n} (0) \geq - C$.
This proves
\begin{equation}
\label{cma-310}
 u_{x_n x_n} (0) \leq \frac{C}{\eta (0)} \leq \frac{C}{c_2}.
\end{equation}

So we have an {\em a priori}
upper bound for all eigenvalues of $\{\fg_{i\bj} (0)\}$.
Therefore they must admit a positive lower bound. Consequently,
by equation~ \eqref{cma2-M10},
\[ m_0 \geq \frac{\sum\fg_{i\bi} (0)}{\fg_{n\bn} (0)} \geq 1 + c_0. \]
This completes the proof of \eqref{cma-37}.


\bigskip

\small
\begin{thebibliography}{99}


\bibitem{Aubin78}
T. Aubin,
{\em \'Equations du type Monge-Amp\`ere sur les vari\'et\'es k\"ahl\'eriennes
compactes}, (French)  Bull. Sci. Math. (2) {\bf 102} (1978), 63--95.

\bibitem{Blocki09}
Z. Blocki,
{\em A gradient estimate in the Calabi-Yau theorem},
Math. Ann. {\bf 344} (2009), 317--327.

\bibitem{Blocki}
Z. Blocki,
{\em On geodesics in the space of K\"ahler metrics},
preprint.

\bibitem{CKNS}
L. A. Caffarelli, J. J. Kohn, L. Nirenberg and J. Spruck, {\em The
Dirichlet problem for nonlinear second-order elliptic equations II.
Complex Monge-Amp\`{e}re and uniformly elliptic equations},
{ Comm. Pure Applied Math.} {\bf 38} (1985), 209--252.

\bibitem{Calabi56}
E. Calabi,
{\em The space of K\"ahler metrics},
Proc. ICM, Amsterdam 1954, Vol. 2, 206--207, North-Holland, Amsterdam, 1956.

\bibitem{Chen00}
X.-X. Chen,
{\em The space of K\"ahler metrics},
J. Differential Geom. {\bf 56} (2000), 189--234.


\bibitem{Chen04}
X.-X. Chen,
{\em A new parabolic flow in K\"ahler manifolds},
Comm. Anal. Geom. {\bf 12} (2004), 837--852.


\bibitem{Cherrier87}
P. Cherrier,
{\em Equations de Monge-Amp\`ere sur les vari\'et\'es hermitiennes compactes},
Bull. Sci. Math. {\bf 111} (1987), 343--385.



\bibitem{CH99}
P. Cherrier and A. Hanani,
{\em Le probl\`eme de Dirichlet pour des \'equations de Monge-Amp\`ere
en m\'etrique hermitienne},
 Bull. Sci. Math.  {\bf 123}  (1999), 577--597.

\bibitem{DK}
S. Dinew and S. Kolodziej,
{\em Liouville and Calabi-Yau type theorems for complex Hessian equations},
arXiv: 1203.3995.


\bibitem{Donaldson99}
S. K. Donaldson,
{\em Symmetric spaces, K\"ahler geometry and Hamiltonian dynamics},
Northern California Symplectic Geometry Seminar, 13--33,
Amer. Math. Soc. Transl. Ser. 2, {\bf 196}, Amer. Math. Soc., Providence, RI, 1999.

\bibitem{Donaldson99a}
 S. K. Donaldson,
{\em Moment maps and diffeomorphisms},
Asian J. Math. {\bf 3} (1999), 1--16.

\bibitem{Donaldson06}
 S. K. Donaldson,
{\em Two-forms on four-manifolds and elliptic equations},
in Inspired by Chern, World Scientific, 2006.

\bibitem{FLM11}
H. Fang, M.-J. Lai and X.-N. Ma,
{\em On a class of fully nonlinear flows in K\"ahler geometry}
J. Reine Angew. Math. {\bf 653} (2011), 189--220.

\bibitem{Gill11}
M. Gill,
{\em Convergence of the parabolic complex Monge-Amp\`ere equation on
compact Hermitian manifolds}, Comm. Anal. Geom. {\bf 19} (2011), 277--304.


\bibitem{Guan98a}
B. Guan,
{\em The Dirichlet problem for Monge-Amp\`ere equations in non-convex
domains and spacelike hypersurfaces of constant Gauss curvature},
Trans. Amer. Math. Soc. {\bf 350} (1998), 4955--4971.

\bibitem{Guan98b}
B. Guan, {\em The Dirichlet problem for complex Monge-Amp\`ere
equations and regularity of the pluri-complex Green function},
Comm. Anal. Geom. {\bf 6} (1998), 687--703. {\em A correction},
 {\bf 8} (2000), 213--218.

\bibitem{GL10}
B. Guan and Q. Li,
{\em Complex Monge-Amp\`ere equations and totally real submanifolds},
Adv. Math. {\bf 225} (2010) 1185--1223.

\bibitem{GL12}
B. Guan and Q. Li,
{\em A Monge-Amp\`ere type fully nonlinear equation on Hermitian manifolds},
Disc. Cont. Dynam. Syst. B {\bf 17} (2012) 1991--1999.

\bibitem{GS93}
B.~Guan and J.~Spruck,
{\em Boundary value problem on $\bfS^n$ for surfaces of constant Gauss
 curvature},
{ Annals of Math.} {\bf 138} (1993), 601--624.


\bibitem{GuanPF02}
P.-F. Guan,
{\em Extremal  functions related to intrinsic norms},
Ann. of Math. {\bf 156} (2002), 197--211.

\bibitem{GuanPF08}
P.-F. Guan,
{\em Remarks on the homogeneous complex Monge-Amp\`ere equation},
Complex Analysis, Trends in Math., Springer Basel AG. (2010), 175--185.

\bibitem{GuanPF}
P.-F. Guan,
{\em A gradient estimate for complex Monge-Amp\`ere equation},
unpublished.

\bibitem{GTW99}
P.-F. Guan, N. Trudinger and X.-J. Wang,
{\em On the Dirichlet problem for degenerate Monge-Amp\`ere equations},
Acta Math. {\bf 182} (1999), 87--104.

\bibitem{Hanani96a}
A. Hanani,
{\em \'Equations du type
 de Monge-Amp\`ere sur les vari\'et\'es hermitiennes compactes},
J. Funct. Anal. {\bf 137} (1996), 49--75.

\bibitem{LaNave-Tian}
G. La Nave and G. Tian,
{\em Soliton-type metrics and K\"ahler--Ricci flow on symplectic quotients},
preprint, arXiv: 0903.2413v1.

\bibitem{LiSY04}
S.-Y. Li,
{\em On the Dirichlet problems for symmetric function equations of the
eigenvalues of the complex Hessian},
Asian J. Math. {\bf 8} (2004), 87--106.

\bibitem{PS}
D.H. Phong and J. Sturm,
{\em The Dirichlet problem for degenerate complex Monge-Amp\`ere
equations}, Comm. Anal. Geom. {\bf 18} (2010), 145--170.

\bibitem{SW08}
J. Song and B. Weinkove,
{\em On the convergence and singularities of the J-flow with applications
to the Mabuchi energy},
Comm. Pure Appl. Math. {\bf 61} (2008), 210--229.

\bibitem{ST10}
J. Streets and G. Tian,
{\em A parabolic flow of pluriclosed metrics},
Int. Math. Res. Not. IMRN {\bf 2010} (2010), 3101--3133.

\bibitem{ST11a}
J. Streets and G. Tian,
{\em  Hermitian curvature flow},
J. Eur. Math. Soc. {\bf 13} (2011), 601--634.

\bibitem{ST11b}
J. Streets and G. Tian,
{\em Regularity theory for pluriclosed flow},
C. R. Math. Acad. Sci. Paris {\bf 349} (2011), 1--4.

\bibitem{ST12}
J. Streets and G. Tian,
{\em Generalized K\"ahler geometry and the pluriclosed flow},
Nuclear Phys. B {\bf 858} (2012), 366--376.




\bibitem{TWv10a}
V. Tosatti and B. Weinkove,
{\em Estimates for the complex Monge-Amp\`ere equation on Hermitian and
balanced manifolds}, Asian J. Math. {\bf 14} (2010), 19--40.

\bibitem{TWv10b}
V. Tosatti and B. Weinkove,
{\em The complex Monge-Amp\`ere equation on compact
Hermitian manifolds}, J. Amer. Math. Soc. {\bf 23} (2010), 1187--1195.

\bibitem{TWv11a}
V. Tosatti, B. Weinkove,
{\em The Calabi-Yau equation on the Kodaira-Thurston manifold},	
J. Inst. Math. Jussieu {\bf 10} (2011), 437--447.

\bibitem{TWv11b}
V. Tosatti, B. Weinkove,
{\em The Calabi-Yau equation, symplectic forms and almost complex structures},
Geometry and analysis. No. 1, 475--493, Adv. Lect. Math. (ALM) 17,
Int. Press, Somerville, MA, 2011.

\bibitem{TWv}
V. Tosatti and B. Weinkove,
{\em On the evolution of a Hermitian metric by its Chern-Ricci form},
arXiv:1201.0312

\bibitem{TWY08}
V. Tosatti, B. Weinkove, and S.-T. Yau,
{\em Taming symplectic forms and the Calabi-Yau equation},
Proc. London Math. Soc. {\bf 97} (2008), 401--424.

\bibitem{Trudinger95}
N. S. Trudinger,
{\em On the Dirichlet problem for Hessian equations},
 Acta Math. {\bf 175} (1995), 151--164.

\bibitem{Weinkove04}
B. Weinkove,
{\em Convergence of the J-flow on K\"ahler surfaces},
Comm. Anal. Geom. {\bf 12} (2004), 949--965.

\bibitem{Weinkove06}
B. Weinkove,
{\em On the J-flow in higher dimensions and the lower boundedness
of the Mabuchi energy},
J. Differential Geom. {\bf 73} (2006), 351--358.

\bibitem{Weinkove07}
B. Weinkove
{\em The Calabi-Yau equation on almost-K\"ahler four-manifolds}
 J. Differential Geom. {\bf 76} (2007), 317--349.

\bibitem{Yau78}
S.-T. Yau,
{\em On the Ricci curvature of a compact K\"ahler manifold and the complex
Monge-Amp\`ere equation. I.}
Comm. Pure Appl. Math. {\bf 31} (1978), 339--411.

\bibitem{Zhang10}
X.-W. Zhang,
{\em  A priori estimate for complex Monge-Amp\`ere equation on Hermitian
manifolds}, Int. Math. Res. Notices {\bf 2010} (2010), 3814--3836.

\end{thebibliography}
\end{document}